\documentclass[11pt]{amsart}
\usepackage{amsmath,amssymb,amsthm,amscd}
\usepackage{tabularx}
\allowdisplaybreaks
\usepackage{amsmath}
\usepackage{amssymb}
\usepackage{amssymb,amsmath,amsthm,amsfonts,latexsym,euscript}
\usepackage{wasysym}
\usepackage{fancyhdr}
\fancypagestyle{plain}
\usepackage{hyperref}
\newtheorem*{theorem*}{Theorem}
\newtheorem{theorem}{Theorem}
\newtheorem{definition}[theorem]{Definition}
\newtheorem{lemma}[theorem]{Lemma}
\newtheorem{corollary}[theorem]{Corollary}

\newtheorem*{remark}{Remark}

\usepackage{amssymb}
\usepackage[T1]{fontenc}
\usepackage[latin1]{inputenc}
\usepackage[usenames,dvipsnames]{xcolor}
\usepackage{graphicx}
\usepackage{setspace}
\usepackage{txfonts}
\usepackage{longtable}
\usepackage{multirow}
\usepackage{blindtext}
\usepackage{ulem}
\usepackage{bm}
\usepackage[margin=1in]{geometry}

\AtEndDocument{{\footnotesize%
 \textsc{G.L. Price, Department of Mathematics, United States Naval Academy, Annapolis, MD., 21402} \par  
  \textit{E-mail address}, G.L.~Price: \texttt{glp@usna.edu} \par
  \addvspace{\medskipamount}
  \textsc{K. Thompson, Department of Mathematics, United States Naval Academy, Annapolis, MD., 21402} \par
  \textit{E-mail address}, K. ~Thompson: \texttt{kthomps@usna.edu} 
 }}

\title{Palindromic Polynomials Over Finite Fields}
\author[G.L.Price]{Geoffrey L. Price}
\thanks{The first author was supported in part by the United States Naval Research Laboratory, Washington, D.C.}
\author[K.Thompson]{Katherine Thompson}
\begin{document}



\begin{abstract}
For any finite field $\mathbb{F}$ and any positive integer $n$ we count the number of monic polynomials of degree $n$ over $\mathbb{F}$ with nonzero constant coefficient and a self-reciprocal factor of any specified degree.  An application is given for systems of linear equations over $\mathbb{F}$ of index $2$.\\
\\
2010 MSC Classification 11T06, 15A06  Keywords: finite fields, self-reciprocal polynomials, partially reciprocal polynomials, linear systems
\end{abstract}
\maketitle

 \section{Introduction and Statement of Results}
 Fix a prime $p$.  Let $ \mathbb F_q$ be the field of $q=p^k$ elements where $k \in \mathbb N$.  Let $f$ be a monic polynomial over $\mathbb F$ of degree $n\geq 1$ and with non-zero constant coefficient.  The \textbf{reciprocal} of $f$ is the degree $n$ polynomial $f^*$ given by $f^*(x) = x^nf(1/x)$.  
 We say that a polynomial $f$ with non-zero constant coefficient is a \textbf{reciprocal polynomial} (or a self-reciprocal polynomial, or a palindrome polynomial) if $f=f^*$.  
 
Given a nonnegative integer $n$ let $z(n)$ denote the number of monic polynomials of degree $n$ over $\mathbb F_q$ with nonzero constant coefficient that have no reciprocal factors other than the constant function $1$. Our main result is:
\begin{theorem}\label{main}
Fix $q=p^k$ for $p$ prime and $k \in \mathbb N$. Then
\begin{itemize}
\item[(i)] $z(0)=1$;
\item[(ii)] $z(1) = q-2$;
\item[(iii)] for $n \geq 2$, $z(n)= k(q)q^{n-1} + c(q,n)$ where $k(q) = \frac{(q-1)^2}{q+1}$ and $c(q,n)=(-1)^{n+1}2(\frac{q-1}{q+1})$.
\end{itemize}
\end{theorem}
Remark: In Theorem $4.4$ of \cite{Price} (see also \cite{Price2}) the first author proved Theorem \ref{main} for the case $q=2$. This result was used in the classification (up to conjugacy) of the binary shifts on the hyperfinite $II_1$ factor.
%
 
 \section{Proof}
 
\begin{lemma} For $n\in \mathbb N$ let $t(n)$ be the number of monic polynomials of degree $n$ over $\mathbb F_q$ with non-zero constant coefficient.  Then $t(n) = (q-1)q^{n-1}$.
\end{lemma}
\begin{proof} There are $q-1$ choices for the constant term and $q$ choices for each of the degree $d$ terms, for $1\leq d \leq n-1$.
\end{proof}

We will use the notation $pr(n)$ to denote the number of \textbf{partially reciprocal} polynomials of degree $n$, i.e. the monic polynomials of degree $n$ with nonzero constant term which have a reciprocal factor of degree at least $2$.  

\begin{lemma}\label{helpful}  For $n\geq 2$, $t(n) - pr(n) = z(n)+z(n-1)$.
\end{lemma}
\begin{proof}
Clearly $t(n)-pr(n)$ is the number of monic polynomials with nonzero constant coefficient which have a reciprocal factor of degree at most $1$.  Note that such a polynomial $f$ either has no palindrome factors or has $x+1$ as its only palindrome factor.  If $x+1$ is the only palindrome factor of $f$, then $g(x)=\frac{1}{x+1}f(x)$ is a polynomial of degree $n-1$ with no palindrome factors.  Conversely, if $g$ is a monic polynomial over $\mathbb F_q$ of degree $n-1$ with no palindrome factors, then $f(x)=(x+1)g(x)$ is a polynomial with $x+1$ as its only palindrome factor, so there is a one-to-one correspondence between polynomials of degree $n-1$ with no palindrome factors and polynomials of degree $n$ with just a single linear palindrome factor.  The result follows.
\end{proof}

\begin{lemma}  For each $n\in \mathbb N$ let s(n) be the number of monic reciprocal polynomials. For any $j\in \mathbb N \cup \{0\}$, $s(2j) = q^j = s(2j+1)$.  
\end{lemma} 
\begin{proof}
Direct, and left to the reader.
\end{proof}      

We now proceed with the proof of Theorem \ref{main}. 
\begin{proof}
\begin{itemize}
\item[(i)] The only monic constant is $f(x)=1$, so $z(0)=1$ trivially.
\item[(ii)] Let $f(x) = x+a$ with $a \in \mathbb F_q$. We exclude $a=0,1$, leaving $q-2$ polynomials. So $z(1)=q-2$.
\item[(iii)] We proceed by induction on $n \geq 2$. \\
\\
For $n=2$, we have $pr(2) = z(0)s(2)=1\cdot q = q$, so by Lemma \ref{helpful}, $t(2)-pr(2) = (q-1)q - q = q^2 -2q$.  So then, since $z(1)=q-2$ and $z(1)+z(2)=t(2)-pr(2)$, we have 

\begin{eqnarray*}
z(2)&=&q^2-2q - (q-2) \\
&=& (q-1)(q-2)\\
&= &\frac{(q-1)(q^2-q-2)}{q+1}\\
&= &\frac{(q-1)[q(q-1) - 2]}{q+1}\\
& = & \dfrac{(q-1)^2}{q+1} \cdot q^{2-1} + (-1)^{2+1}2 \left( \dfrac{q-1}{q+1}\right)\\
&=&k(q)q^{2-1} + c(q,2),
\end{eqnarray*}
which completes the base case.\\
\\
For the inductive step, we first suppose $n$ is odd, and write $n=2m+1$ for $m \in \mathbb N$, By assumption $z(2m)=k(q)q^{2m-1}+c(q,2m)$. Then
\[
 pr(n) = pr(2m+1) = s(2m+1)z(0)+s(2m)z(1)+\cdots + s(2)z(2m-1) = I + II, 
 \]
where
\begin{eqnarray*}
I &= &\sum_{\ell=0}^{m-1} z(2\ell)s(2(m-\ell)+1)\\
    &= &z(0)s(2m+1) +  \sum_{\ell=1}^{m-1} z(2\ell)s(2(m-\ell)+1)\\
    &= &1\cdot q^m + \sum_{\ell=1}^{m-1} \left(k(q)q^{2\ell -1} -2\left(\frac{q-1}{q+1} \right) \right)q^{m-\ell}\\
    &= &q^m + \sum_{\ell=1}^{m-1} \left( k(q)q^{m+\ell -1} -2\left( \dfrac{q-1}{q+1}\right)p^{m-\ell} \right)\\
    &= &q^m + k(q) (q^m + \cdots + q^{2m-2}) -2\left( \dfrac{q-1}{q+1}\right)(q^{m-1}+\cdots + q)\\
    &= &q^m + k(q) q^m\left(\frac{q^{m-1}-1}{q-1}\right) -2 \left( \dfrac{q-1}{q+1}\right)q(\frac{q^{m-1}-1}{q-1})\\
    &= &q^m +k(q) \left(\frac{q^{2m-1}-q^m}{q-1}\right)- 2 \left( \dfrac{q-1}{q+1}\right)\left(\frac{q^m-q}{q-1}\right),
\end{eqnarray*}
and
\begin{eqnarray*}
 II &=& s(2m)z(1) + s(2m-2)z(3) + \cdots + s(2)z(2m-1)\\
     &=& s(2m)z(1) + \sum_{\ell =0}^{m-2} z(2\ell + 3)s(2(m-\ell -1))\\
     &=& (q-2)q^{m} + \sum_{\ell =0}^{m-2} \left(k(q)q^{2\ell +2}+2\left( \dfrac{q-1}{q+1}\right)\right)q^{m-\ell -1}\\
     &=& (q-2)q^{m} + k(q)(q^{m+1} + \cdots + q^{2m-1}) + 2 \left( \dfrac{q-1}{q+1}\right)(q^{m-1}+\cdots + q)\\
     &=& (q-2)q^{m} +k(q)q^{m+1}\left(\frac{q^{m-1}-1}{q-1}\right) + 2 \left( \dfrac{q-1}{q+1} \right)q\left(\frac{q^{m-1}-1}{q-1}\right)\\
     &=& (q-2)q^{m}+ k(q) \left(\frac{q^{2m}-q^{m+1}}{q-1}\right) +2 \left( \dfrac{q-1}{q+1}\right)\left(\frac{q^{m}-q}{q-1}\right).
\end{eqnarray*}
Then

\begin{eqnarray*}
pr(2m+1) &=& I + II\\
& = & q^m +k(q) \left(\frac{q^{2m-1}-q^m}{q-1}\right)- 2 \left( \dfrac{q-1}{q+1}\right)\left(\frac{q^m-q}{q-1}\right) + \\
& & (q-2)q^{m}+ k(q) \left(\frac{q^{2m}-q^{m+1}}{q-1}\right) +2 \left( \dfrac{q-1}{q+1}\right)\left(\frac{q^{m}-q}{q-1}\right)\\
 &=& q^m +k(q) \left(\frac{q^{2m-1}-q^m}{q-1}\right) + (q-2)q^{m}+ k(q) \left(\frac{q^{2m}-q^{m+1}}{q-1}\right)\\
 &=& q^m + \left(\frac{q-1}{q+1}\right)(q^{2m-1}-q^m) + (q-2)q^m + \left(\frac{q-1}{q+1}\right)(q^{2m}-q^{m+1})\\
 & = & (q-1)q^m + \left( \dfrac{q-1}{q+1}\right)(q^{2m-1}-q^m+q^{2m}-q^{m+1})\\
 & = & \dfrac{q-1}{q+1}(q^{m+1}+q^m) +  \dfrac{q-1}{q+1}(q^{2m-1}-q^m+q^{2m}-q^{m+1})\\
 & = & (q-1)q^{2m-1}.
\end{eqnarray*}
But now by Lemma \ref{helpful}
\begin{eqnarray*}
z(2m+1) & = & t(2m+1)-pr(2m+1)-z(2m) \\
& = & (q-1)q^{2m}-(q-1)q^{2m-1} - \left( \dfrac{(q-1)^2}{q+1}\cdot q^{2m-1} -2\cdot \dfrac{q-1}{q+1}\right)\\
& = & (q-1)q^{2m-1}\left( q-1-\frac{q-1}{q+1}\right) + 2 \left( \frac{q-1}{q+1} \right)\\
& = & (q-1)^2q^{2m-1} \left( 1- \frac{1}{q+1} \right) + 2 \left( \frac{q-1}{q+1}\right)\\
& = & k(q) q^{2m}+c(q,2m+1).
\end{eqnarray*}


\noindent We now suppose that $n$ is even, and write $n=2m$ where $m \in \mathbb N$. Assume $z(\ell)= k(q)q^{\ell-1}+c(q,n)$ for  
$\ell = 2,\dots,2m-1$.  We have 
\[ pr(2m)= \sum_{k=0}^{2m-2} z(k)s(2m-k) = I + II, \]
where
\begin{eqnarray*}
  I &=& \sum_{\ell=0}^{m-1} z(2\ell)s(2(m-\ell))\\
    &=& z(0)s(2m) +  \sum_{\ell=1}^{m-1} z(2\ell)s(2(m-\ell))\\
    &=& 1\cdot q^m + \sum_{\ell=1}^{m-1} \left( k(q) \cdot q^{2\ell -1} -2 \left( \dfrac{q-1}{q+1}\right) \right)q^{m-\ell}\\
    &=& q^m + \sum_{\ell=1}^{m-1} \left( k(q)q^{m+\ell -1} -2 \left( \dfrac{q-1}{q+1}\right)q^{m-\ell}\right)\\
    &=& q^m +k(q) q^m(1+\cdots + q^{m-2}) -2 \left( \dfrac{q-1}{q+1}\right)q(1+\cdots + q^{m-2})\\
    &=& q^m + k(q) q^m\left( \frac{q^{m-1}-1}{q-1}\right) -2 \left( \dfrac{q-1}{q+1} \right)q\left( \frac{q^{m-1}-1}{q-1}\right)\\
    &=& q^m +k(q) \left( \frac{q^{2m-1}-q^m}{q-1}\right)- 2 \left( \dfrac{q-1}{q+1}\right)\left(\frac{q^m-q}{q-1}\right),
\end{eqnarray*}
and
\begin{eqnarray*}
  II &=& \sum_{\ell = 1}^{m-1} z(2(\ell -1)+1)s(2(m-\ell)+1)\\
     &=& z(1)s(2(m-1)+1) + \sum_{\ell =2}^{m-1}  z(2(\ell -1)+1)s(2(m-\ell)+1)\\
     &=& (q-2)q^{m-1} + \sum_{\ell =2}^{m-1} \left(k(q)q^{2\ell -2}+2 \left( \dfrac{q-1}{q+1}\right)\right)q^{m-\ell}\\
     &=& (q-2)q^{m-1} + k(q)(q^m + \cdots + q^{2m-3}) + 2 \left( \dfrac{q-1}{q+1}\right)(q^{m-2}+\cdots + q)\\
     &=& (q-2)q^{m-1} +k(q)q^m(1+\cdots + q^{m-3})+ 2 \left( \dfrac{q-1}{q+1}\right)q(1 + \cdots + q^{m-3})\\
     &=& (q-2)q^{m-1} +k(q)q^m(\frac{q^{m-2}-1}{q-1}) + 2 \left( \dfrac{q-1}{q+1}\right)q\left(\frac{q^{m-2}-1}{q-1}\right)\\
     &=& (q-2)q^{m-1}+ k(q)(\frac{q^{2m-2}-q^m}{q-1}) +2 \left( \dfrac{q-1}{q+1}\right)\left(\frac{q^{m-1}-q}{q-1}\right).
\end{eqnarray*}
Then
\begin{eqnarray*}
  pr(2m) &=& I + II\\
  & = & q^m +k(q) \left( \frac{q^{2m-1}-q^m}{q-1}\right)- 2 \left( \dfrac{q-1}{q+1}\right)\left(\frac{q^m-q}{q-1}\right) + (q-2)q^{m-1}+ k(q)\left(\frac{q^{2m-2}-q^m}{q-1}\right) +2 \left( \dfrac{q-1}{q+1}\right)\left(\frac{q^{m-1}-q}{q-1}\right)\\
  & = & q^m +(q-2)q^{m-1}+ k(q) \left( \dfrac{q^{2m-1}+q^{2m-2}-2q^m}{q-1} \right) - 2 \left( \dfrac{q-1}{q+1}\right) \left( \dfrac{q^m-q^{m-1}}{q-1}\right)\\
  & = & 2(q^m-q^{m-1})+ \left( \dfrac{q-1}{q+1}\right) (q^{2m-1}+q^{2m-2}-2q^m) - \dfrac{2(q^m-q^{m-1})}{q+1}\\
  & = & \dfrac{2(q^m-q^{m-1})(q+1)}{q+1}- \dfrac{2(q^m-q^{m-1})}{q+1}+ \dfrac{q^{2m}-2q^{m+1}-q^{2m-2}+2q^m}{q+1}\\
  & = & \dfrac{2q^{m+1}-2q^m+q^{2m}-2q^{m+1}-q^{2m-2}+2q^m}{q+1} \\
   & = & q^{2m-2}(q-1).
\end{eqnarray*}
But again by Lemma \ref{helpful} we have
\begin{eqnarray*}
z(2m) & = & t(2m)-pr(2m)-z(2m-1) \\
& = & (q-1)q^{2m-1}-(q-1)q^{2m-2}-\left( k(2m-1)q^{2m-2}+2 \left( \dfrac{q-1}{q+1}\right) \right)\\
& = & (q-1)^2q^{2m-2}-\left( \dfrac{(q-1)^2}{q+1}\right)q^{2m-2}- 2 \left( \dfrac{q-1}{q+1}\right)\\
& = & \left( \dfrac{(q-1)^2}{q+1}\right)q^{2m-2}-2 \left( \dfrac{q-1}{q+1} \right) \\
 & = & k(q)q^{2m-1}+c(q, 2m).
\end{eqnarray*}
\end{itemize}
\end{proof}
\begin{corollary}  Let $j$ and $n$ be positive integers with $1 \leq j \leq n$.  Let $p(n,j)$ be the number of monic polynomials $f$ over $\mathbb{F}_q$ with nonzero constant coefficient such that $f$ has a self-reciprocal factor of degree $j$ and no higher.  Then $p(n,j)=z(n-j)s(j)$.
\end{corollary}
\section{Application to Systems of Linear Equations}
As mentioned above, our counting result was established in \cite{Price} for $\mathbb F_2$. Here $z(0)=1, z(1)= z(2)= 0$ and for $n\geq 3$, $z(n) = \frac13(2^{n-1})+\frac23$ (for $n$ odd) and $\frac13(2^{n-1})-\frac23$ (for $n$ even).  We present an application of  these results to a family of linear systems of equations over $\mathbb{F}_2$. This application is related to the classification problem in \cite{Price} mentioned in Section 1 above. 

\begin{definition} Let $(a)=(\dots, a_{-2}, a_{-1}, a_0,a_1,a_2,\dots)$ be a sequence in $\mathbb{F}_2$ with $a_0 = 0$ and with $a_j=a_{-j}$.  We shall say that $(a)$ has index $2$ if one of the following holds:
\begin{enumerate}
\item $a_{-1}=a_1=1$, and $a_j =0$ otherwise. 
\item $a_0=0$ and $a_j=1$ otherwise.
\item For some $m\geq 2$ there is a vector $\vec{k}=[k_0,k_1,\dots,k_{m}]$, with $k_0 = k_{m}=1$ 
satisfying the following infinite linear system of equations:
\begin{alignat*}{4}
\notag 1 =& k_0a_1 + k_1a_0 + k_2a_1+\cdots + k_ma_{m-1} \\
\notag 0 =& k_0a_2 + k_1a_1 + k_2a_0+\cdots + k_ma_{m-2} \\
\notag 0 =& k_0a_3 + k_1a_2 + k_2a_1 + \cdots+ k_ma_{m-3}\\
\notag 0 =& k_0a_4 + k_1a_3 + k_2a_2 + \cdots + k_ma_{m-4}\\
\notag \vdots \\
\end{alignat*}
\end{enumerate}
\end{definition}
Note that $(1)$ may be rewritten (with $\vec{k} =[k_0] = [1]$ and $(a)=(0100\cdots)$ as
\begin{alignat*}{4}
\notag 1 =& k_0a_1\\
\notag 0 =& k_0a_2\\
\notag 0 =& k_0a_3\\
\notag 0 =& k_0a_4\\
\notag \vdots \\
\end{alignat*} 
and $(2)$ may be rewritten (with $\vec{c} =[k_0,k_1]=[1,1]$ and $(a)=(0111\cdots)$ as
\begin{alignat*}{4}
\notag 1 =& k_0a_1+k_1a_0\\
\notag 0 =& k_0a_2+k_1a_1\\
\notag 0 =& k_0a_3+k_1a_2\\
\notag 0 =& k_0a_4+k_1a_3\\
\notag \vdots \\
\end{alignat*}  
One can show that the sequence $a_{m-1},a_{m-2},\dots,a_1,a_0,a_1,a_2,\dots$ in (3) is not periodic, whereas \\$a_{m-2},a_{m-3},\dots,a_1,a_0,a_1,a_2,\dots$ is periodic, see \cite{LN}[Theorem 6.11].  From the proof of Theorem 3.3 of \cite{Price} it follows that a vector $\vec{k}=[k_0,k_1,\dots,k_{m}]$ with $m\geq 2$ satisfies a system of the form $(3)$ for at most one sequence $(a)$.  From the proof of Lemma 3.1 of \cite{Price}, a necessary condition for a vector $\vec{k}$ to satisfy $(3)$ is that the polynomial $k_0 + k_1x + \cdots + k_mx^m$ either has no palindrome factors or the single palindrome factor $1+x$ over $\mathbb{F}_2$.  There are exactly $z(m) + z(m-1)$ such polynomials, with $z(m)$ polynomials of degree $m$ and constant coefficient $1$ having no palindrome factors, and $z(m-1)$ such polynomials of degree $m$ having $1+x$ as its only palindrome factor.  Therefore there are $z(m)+z(m-1) = \frac13(2^{m-1}) + \frac13(2^{m-2}) = 2^{m-2}$ such polynomials.  By the proof of Theorem 3.4 of \cite{Price} there are at least $2^{m-2}$ vectors $\vec{k}=[k_0,k_1,\dots,k_m]$ satisfying (3).  Therefore we have the following result (cf. \cite{Price}[Theorem 5.1]).

\begin{theorem}  For $m\geq 3$ there are exactly $2^{m-2}$ linear systems of index $2$ satisfying $(3)$ over $\mathbb{F}_2$.
\end{theorem}

\begin{remark} Note that every degree $2$ polynomial with constant coefficient $1$ is a palindrome so that there are no systems (3) with $m=2$.
\end{remark}

\begin{remark} There is a notion of a linear system of equations over $\mathbb{F}_2$ of index $j$ with $j>2$, but the analysis is considerably more complicated, as there is no longer a one-to-one correspondence between vectors $\vec{k}$ and sequences $(a)$ (see \cite{Price2} for related results).
\end{remark}

\end{document}